\documentclass{amsart}

\DeclareMathSymbol{\twoheadrightarrow}  {\mathrel}{AMSa}{"10}

\def\Q{{\mathbb Q}}
\def\Z{{\mathbb Z}}
                             \def\NN{{\mathbb N}}
\def\C{{\mathbb C}}

\def\RR{{\mathbb R}}
\def\F{{\mathbb F}}
\def\P{{\mathbb P}}
             \def\PP{{\mathfrak P}}

             \def\Fr{\mathrm{Fr}}

\def\f{{\tilde F}}
                     \def\f0{{\mathfrak f}}

                                            \def\Norm{\mathrm{Norm}}

\def\A8{{\mathbf A}_8}

\def\RR{{\mathfrak R}}

\def\Gal{\mathrm{Gal}}

                              \def\ord{\mathrm{ord}}

\def\End{\mathrm{End}}
\def\Aut{\mathrm{Aut}}

\def\I{\mathrm{Id}}

              \def\L{{\mathcal L}}

\def\ST{{\mathbf S}}

\def\fchar{\mathrm{char}}

                                                \def\ssL{\mathrm{ss}}

\def\A{\mathbf{A}}

\def\B{{\mathfrak B}}

\def\dim{\mathrm{dim}}
                           \def\rk{\mathrm{rk}}
                           \def\length{\mathrm{length}}
                           \def\Sl{\mathrm{Slp}}
\def\Oc{{\mathcal O}}

\newtheorem{thm}{Theorem}[section]
\newtheorem{lem}[thm]{Lemma}
\newtheorem{cor}[thm]{Corollary}

\theoremstyle{definition}
\newtheorem{defn}[thm]{Definition}
\newtheorem{ex}[thm]{Example}
\newtheorem{rem}[thm]{Remark}

\newtheorem{rems}[thm]{Remarks}
        
\title[Eigenvalues of Frobenius Endomorphisms of Abelian varieties]{Eigenvalues of Frobenius Endomorphisms of Abelian varieties of low dimension}

\author[Yuri\ G.\ Zarhin]{Yuri\ G.\ Zarhin}
\address{Department of Mathematics, Pennsylvania State University,
University Park, PA 16802, USA}
\address{Department of Mathematics, The Weizmann Institute of Science,
 P.O.B. 26,  Rehovot 7610001, Israel}

\thanks{This work was partially supported by the Simons Foundation (grant \#246625 to Yuri Zarkhin).}
\email{zarhin\char`\@math.psu.edu}
\thanks{}
\begin{document}
\begin{abstract}
In this paper we discuss {\sl nontrivial} multiplicative relations among  eigenvalues of  Frobenius endomorphisms of abelian varieties   over finite fields. (The trivial relations are provided by the ``Riemann Hypothesis" that was proven by A. Weil.) We classify all abelian varieties over finite fields of dimension $\le 3$ that admit the nontrivial  relations.
\end{abstract}
\maketitle

\section{Introduction}

There is a lasting interest  in the study of multiplicative relations between eigenvalues of the
 Frobenius endomorphism $\Fr_X$ of an  abelian varietiy $X$ over a finite field $k=\F_q$ of characteristic $p$ (where $q$ is a power of $p$).
A  {\sl nontrivial} multiplicative relation between the eigenvalues  gives rise to an {\sl exotic}  Tate class on a certain self-product of $X$ \cite{ZarhinK3,LenstraZarhin,ZarhinEssen}. (Here  exotic means that this class cannot be presented as a linear combination of products of divisor classes.) These relations are important in the study of $\ell$-adic representations attached to abelian varieties over global fields \cite{ZarhinIzv79,ZarhinInv79}. In particular, they play a crucial role in Serre's theory of {\sl Frobenius tori} \cite{SerreRibet}.
 On the other hand, the absence of nontrivial multiplicative relations between the eigenvalues of most jacobians over finite fields  is viewed as an analogue of conjectures of $\Q$-linear independence of ordinates of zeros of L-functions over number fields \cite{K}. The absence of these relations
  for the jacobian of a given  curve $\mathcal{C}$ over  $\F_q$  was used in \cite{Sh} in order to study an asymptotic behavior of the {\sl normalized error term} in Weil's formula
for the number
 of points of $\mathcal{C}$ in degree $n$ extensions $\F_{q^n}$ of the ground field.

In this paper we study  the nontrivial multiplicative relations for abelian varieties of small dimension.  
 Our main tool, as in \cite{ZarhinIzv79,ZarhinInv79,ZarhinK3,LenstraZarhin,ZarhinEssen}, is the multiplicative group $\Gamma(X,k)$ generated by the set  $R_X$ of eigenvalues of $\Fr_X$.   Recall that $\alpha \mapsto q/\alpha$ is a permutation of $R_X$ and notice that
$$q^{-1}\left(\frac{q}{\alpha}\right)^2=\left(q^{-1}\alpha^2\right)^{-1}.$$
This implies that if $e:R_X \to \Z$ is an integer-valued function such that
$$e(\alpha)=e(q/\alpha) \ \forall
\alpha \in R_X$$ then
$$\prod_{\alpha\in R_X}\left(q^{-1}\alpha^2\right)^{e(\alpha)}=1.$$

 Assuming that $k$ is {\sl sufficiently large} with respect to $X$, i.e., $\Gamma(X,k)$ does {\sl not} contain {\sl nontrivial} roots of unity, we say that $X$ is {\sl neat} (see \cite[Sect. 3]{ZarhinEssen} and Sect. \ref{neat} below) if it enjoys the following property.

{\sl If $e:R_X \to \Z$ is an integer-valued function such that
$$\prod_{\alpha\in R_X}\left(q^{-1}\alpha^2\right)^{e(\alpha)}=1$$ then $e(\alpha)=e(q/\alpha)$ for all}
$\alpha \in R_X$.

Notice that $X$ is neat if and only if every Tate class on each  self-product of $X$  can be presented as a linear combination of products of divisor classes  \cite{ZarhinK3,LenstraZarhin,ZarhinEssen}. 
 (In particular, the Tate conjecture holds true for all self-products of $X$.)  
An analogy with the Hodge conjecture  for complex abelian varieties \cite{MZ}
 suggests that up to dimension 3  all abelian varieties over finite fields should be neat.  
However, it turns out that there are non-neat abelian threefolds (see below).

Our main result is the following statement.

\begin{thm}
\label{main}
Suppose that $1 \le \dim(X) \le 3$ and $k$ is sufficiently large with respect to $X$.  Then $X$ is not neat if and only if it enjoys all of the following three properties.

\begin{itemize}
\item[(i)]
$X$ is abslolutely simple, all endomorphisms of $X$ are defined over $k$ and its endomorphism algebra $\End^0(X)$ is a sextic CM-field that is generated by $\Fr_X$.
\item[(ii)]
 $\End^0(X)$   contains an imaginary quadratic subfield $B$ that enjoys the following property. If
$$\Norm: \End^0(X) \to B$$
is the norm map corresponding to the cubic field extension $\End^0(X)/B$ then
$$\Norm\left(q^{-1} \Fr_X^2\right)=1.$$
\item[(iii)] $X$ is almost ordinary, i.e. the set of slopes of its Newton polygon is $\{0,1/2,1\}$ and $\length(1/2)=2$.
\end{itemize}
\end{thm}

\begin{rems}
Let $X$ and $B$ satisfy the conditions (i)-(iii) of Theorem \ref{main}. Let us fix an embedding $B \subset \C$ of the imaginary quadratic field $B$ into the field $\C$ of complex numbers.
\begin{itemize}
\item
 Let
$$\sigma_1,\sigma_2, \sigma_3: \End^0(X)\hookrightarrow \C$$
be the distinct embeddings of sextic $\End^0(X)$ to $\C$ that act as the identity map on $B$. Let us put
$$\alpha_1=\sigma_1(\Fr_X), \ \alpha_2=\sigma_2(\Fr_X), \ \alpha_3=\sigma_3(\Fr_X).$$
Then $\alpha_1, \alpha_2, \alpha_3$ are distinct eigenvalues of $\Fr_X$, the set $R_X$ consists of six distinct elements
$$\{\alpha_1,\alpha_2,\alpha_3; \ q/\alpha_1,q/\alpha_2,q/\alpha_3\}$$
 and
$$1=\Norm(q^{-1}\Fr_X^2)=\prod_{i=1}^3q^{-1}\alpha_i^2.$$
In particular,
$$q^3=(\alpha_1 \alpha_2 \alpha_3)^2.$$
Notice that the set $\Phi=\{\sigma_1,\sigma_2, \sigma_3\}$ is a CM-type of the sextic CM-field $\End^0(X)$, which is {\sl not} primitive. (See \cite[p. 406]{Ogg} for a concise definition of a primitive CM-type.)
\item
Since $X$ is absolutely simple,  $\End^0(X)$ is a field and  $X$ is {\sl not} ordinary,  it follows from \cite[Th. 3.6.ii)]{Gon} that $p$ does {\sl not} split completely in $\End^0(X)$, i.e., the tensor product $\End^0(X)\otimes_{\Q}\Q_p$ is {\sl not} isomorphic to a direct sum of six copies of $\Q_p$.
\item
Since $\End^0(X)$ is a  sextic CM-field, it is a totally imaginary quadratic extension of a certain totally real cubic field $K$. Clearly, $B$ and $K$ are linearly disjoint over $K$ and the natural field homomorphism
$$B\otimes K \to \End^0(X), \ x\otimes y \mapsto xy$$
is a field isomorphism.
The cubic extension $K/\Q$ is {\sl not} Galois. Indeed, otherwise it is abelian (even cyclic), the field extension  $\End^0(X)/\Q$ is also a Galois extension and its Galois group coincides with the product $\Gal(B/\Q)\times \Gal(K/\Q)$. In particular, $\Gal(\End^0(X)/\Q)$ is abelian.  By \cite[Th. 3.6.iii]{Gon}, $p$  splits completely in $\End^0(X)$, which is not the case.
\item
For each prime $\ell \ne p$ there exists an {\sl exotic} six-dimensional $\ell$-adic Tate class on $X\times X$ \cite[Sect. 3]{ZarhinEssen}.
\end{itemize}
\end{rems}

\begin{rem}
See \cite[Sect. 4]{ZarhinEssen} for examples of non-neat abelian threefolds constructed by Hendrik Lenstra,Jr. See also Section \ref{exam} below.

Notice that the property to be ordinary is an {\sl open condition} in the moduli space of (polarized) abelian varieties of given dimension in characteristic $p$. Thus Theorem \ref{main} implies that a {\sl typical} abelian threefold is {\sl neat}. On the other hand, one may construct  non-neat ordinary abelian fourfolds, using results of \cite{NootCrelle}; see also Sect. \ref{dim4}.
\end{rem}

The paper is organized as follows. In Section \ref{neat} we express the neatness property of $X$ in terms of the minimal
polynomial $\P_{X,\min}(t)$ of $\Fr_X$. In Section \ref{nonsimple} we review results of \cite{ZarhinEssen}. In Section \ref{newton} we discuss Newton polygons of abelian varieties over finite fields.  Section \ref{dim2} contains a non-existence  result for a certain class of simple abelian surfaces. Section \ref{mainproof} contains the proof of Theorem \ref{main}.
Section \ref{exam} deals with examples.
 In Section \ref{dim4} we discuss certain abelian fourfolds over finite fields.

{\bf Acknowledgements}. I am grateful to  Hendrik Lenstra,Jr,  Frans Oort  and Alice Silverberg for helpful discussions, and to Igor Shparlinski for stimulating questions.
My special thanks go to Tatiana Bandman, whose comments helped to improve the exposition.

This work was started during my stay at the Max-Planck-Institut f\"ur Mathematik (Bonn) in September 2013. Most of this work was done during the academic year 2013/2014 when I was Erna and Jakob Michael Visiting Professor in the Department of Mathematics  at the Weizmann Institute of Science. The hospitality and support of both Institutes are gratefully acknowledged.

\section{Ranks of neat abelian varieties}
\label{neat}
As usual, $\ell$ is a prime different from $p$ and $\NN,\Z,\Z_{\ell},\Q,\C,\Q_{\ell}, \Q_p$ stand for the set of positive integers, the rings of integers and $\ell$-adic integers, and the fields of rational, complex, $\ell$-adic and $p$-adic numbers respectively.  If $z$ is a complex number then we write $\bar{z}$ for its complex-conjugate. Similarly, if $\phi: E \hookrightarrow \C$ is a field embedding then we write $\bar{\phi}$ for the corresponding complex-conjugate field embedding
$$\bar{\phi}: E \hookrightarrow \C, \ x \mapsto \overline{\phi(x)}.$$
We write $\Z_{(\ell)}$ for the subring
$$\Z_{(\ell)}:=\{\frac{a}{b}\mid a \in \Z, \ b \in \Z \setminus \ell\Z\}\subset \Q;$$
we may also view $\Z_{(\ell)}$ as the subring of $\Z_{\ell}$. (Similarly, the subring $\Z_{(p)} \subset \Q$ is defined.)
If $A$ is a finite set then we  write $\#(A)$ for number of its elements. We write $\rk(\Delta)$ for rank of a finitely generated commutative group $\Delta$.
Throughout this paper $k$ is a finite field of characteristic $p$ that consists of $q$ elements, $\bar{k}$ an algebraic closure of $k$ and $\Gal(K)=\Gal(\bar{k}/k)$ the absolute Galois group of $k$. It is well known that the profinite group $\Gal(K)$ is procyclic and the {\sl Frobenius automorphism}
$$\sigma_k: \bar{k} \to \bar{k}, \ x \mapsto x^q$$
is a topological generator of $\Gal(k)$.

Let $X$ be an abelian variety of positive dimension over $k$. We write $\End(X)$ for the ring of its $k$-endomorphisms and $\End^0(X)$ for the corresponding (finite-dimensional semisimple) $\Q$-algebra $\End(X)\otimes\Q$. We write $\Fr_X=\Fr_{X,k}$ for the Frobenius endomorphism of $X$. We have
$$\Fr_X \in \End(X)\subset \End^0(X).$$
By a theorem of Tate \cite[Sect. 3, Th. 2 on p, 140]{Tate1}, the $\Q$-subalgebra $\Q[\Fr_X]$ of $\End^0(X)$ generated by $\Fr_X$ coincides with the center of $\End^0(X)$. In particular, if $\End^0(X)$ is a field then $\End^0(X)=\Q[\Fr_X]$.

 If $\ell$ is a prime different from $p$ then we  write $T_{\ell}(X)$ for the $\Z_{\ell}$-Tate module of $X$ and $V_{\ell}(X)$ for the corresponding $\Q_{\ell}$-vector space
$$V_{\ell}(X)=T_{\ell}(X)\otimes_{\Z_{\ell}}\Q_{\ell}.$$
It is well known \cite[Sect. 18]{Mumford} that $T_{\ell}(X)$ is a free $\Z_{\ell}$-module of rank $2\dim(X)$ that may be viewed as a $\Z_{\ell}$-lattice in the $\Q_{\ell}$-vector space $V_{\ell}(X)$   of dimension $2\dim(X)$.

 By functoriality, $\End(X)$ and $\Fr_X$ acts
  on ($T_{\ell}(X)$ and) $V_{\ell}(X)$; it is well known that the action of $\Fr_X$ coincides with the action of $\sigma_k$. By a theorem of A. Weil \cite[Sect. 19 and Sect. 21]{Mumford}, $\Fr_X$ acts on $V_{\ell}(X)$ as a semisimple linear operator, its characteristic polynomial
$$\P_X(t)=\P_{X,k}(t)=\det (t \I -\Fr_X, V_{\ell}(X)) \in \Z_{\ell}[t]$$
lies in $\Z[t]$ and does not depend on a choice of $\ell$. In addition, all eigenvalues of $\Fr_X$ (which are algebraic integers) have archimedean absolute value equal to $q^{1/2}$.  This means that if
$$L=L_X \subset \C$$
is the splitting field of $\P_X(t)$ and  $$R_X=R_{X,k} \subset L$$ is the set of roots of $P(t)$ then $L$ is a finite Galois extension of $\Q$ such that for every field embedding $L \hookrightarrow \C$ we have $\mid \alpha \mid =q^{1/2}$ for all $\alpha \in R_X$.  Let  $\Gal(L/\Q)$ be the Galois group of $L/\Q$.
 Clearly, $R_X$ is a $\Gal(L/\Q)$-invariant (finite) subset of $L^{*}$. It follows easily that if  $\alpha \in R_X$ then $q/\alpha \in R_X$. Indeed,
$q/\alpha$ is the {\sl complex-conjugate} $\bar{\alpha}$ of $\alpha$. We have
$$q^{-1}\alpha^2=\frac{\alpha}{q/\alpha}.$$

\begin{rem}
\label{multiplicities}
Let $m(\alpha)$ be the multiplicity of the root $\alpha$ of $\P_X(t)$. Then
$$P_{X}(t)=\prod_{\alpha\in R_X}(t-\alpha)^{m(\alpha)}\in \C[t] \eqno(1)$$
and
$$\rk(\End(X))=\sum_{\alpha\in R_X} m(\alpha)^2 \eqno(2)$$
(see \cite[pp. 138--139]{Tate1}, especially (4) and (5)).
Let $\kappa$ be a finite overfield of $k$ of degree $d$ and
 $X^{\prime}=X \times_{k}\kappa$. Then $T_{\ell}(X_{\kappa})$ and $V_{\ell}(X_{\kappa})$ are canonically isomorphic to $T_{\ell}(X)$ and $V_{\ell}(X)$ respectively, $$\Fr_{X_{\kappa}}=\Fr_X^{d}\subset \End(X) \subset \End(X_{\kappa}),$$
 $$R_{X_{\kappa}}=\{\alpha^d\mid \alpha \in R_X\}, \ \P_{X_{\kappa}}(t)=\prod_{\alpha\in R_X}(t-\alpha^d)^{m(\alpha)}.$$

Suppose that $\alpha/\beta$ is {\sl not} a root of unity for  every pair of {\sl distinct} $\alpha, \beta \in R_X$.  This implies that $\alpha^d$ and $\beta^d$ are distinct roots of $\P_{X_{\kappa}}(t)$. It follows that for every $\alpha\in R_X$ the positive integer
$m(\alpha)$ coincides with the multiplicity of root $\alpha^d$ of the polynomial $\P_{X_{\kappa}}(t)$. The formulas (1) and (2) applied to $X_{\kappa}$ give us the equality $\rk(\End(X_{\kappa}))=\rk(\End(X))$, which implies that $\End(X_{\kappa})=\End(X)$, because the quotient $\End(X_{\kappa})/\End(X)$ is torsion-free \cite[Sect. 4, p. 501]{SerreTate}. In particular, if $X$ is simple then it is absolutely simple.
\end{rem}

\begin{rem}
\label{nonP}
Let $\Oc_L$ be the ring of integers in $L$. Clearly, $R_X \subset \Oc_L$. It is also clear that if $\B$ is a maximal ideal in $\Oc_L$ such that $\fchar(\Oc_L/\B) \ne p$  then all elements of $R_X$ are $\B$-adic units.

\end{rem}

 \begin{rem}
 \label{simple}
  Notice that $R_X$ is a $\Gal(L/\Q)$-orbit if and only if $\P_X(t)$ is a power of an irreducible polynomial (over $\Q$), which means that $X$ is isogenous over $k$ to a simple abelian variety over $k$ \cite[Theorem 2(e)]{Tate1}
 (see also \cite[Sect. 5, Th. 5.3 and Remark after it]{OortG}).
\end{rem}

By functoriality, $\End^0(X)$ and $\Q[\Fr_X]$ act on $V_{\ell}(X)$. This action extends by $\Q_{\ell}$-linearity to the embedding of $\Q_{\ell}$-algebras
$$\Q[\Fr_X]\otimes_{\Q}\Q_{\ell}\subset  \End^0(X)\otimes_{\Q}\Q_{\ell}=\End(X)\otimes_{\Q}\Q_{\ell}\subset \End_{\Q_{\ell}}(V_{\ell}(X)).$$

\begin{ex}
\label{simpleCM}
Let us assume that
$$E=\Q[\Fr_X]$$
 is a field. Then it is known \cite[Th. 2.1.1 on p. 768]{Ribet} that $V_{\ell}(X)$ carries the natural structure of a free $E\otimes_{\Q}\Q_{\ell}$-module and this module is free of rank $e=2\dim(X)/[E:\Q]$. It follows that
$$\P_X(t)=[\P_{X,\min}(t)]^e, \ 2\dim(X)=\deg(\P_X)=e\deg(\P_{X,\min})$$
where $\P_{X,\min}(t)$ is the minimal polynomial of the semisimple linear operator
$\Fr_X: V_{\ell}(X)\to V_{\ell}(X)$. Clearly, $\P_{X,\min}(t)$ has integer coefficients, $\P_{X,\min}(\Fr_X)=0 \in \End(X)$ and the natural homomorphism
$$\Q[t]/\P_{X,\min}(t)\Q[t] \to \Q[\Fr_X], \  t \mapsto \Fr_X+\P_{X,\min}(t)\Q[t]$$
is a field isomorphism.  In particular, $\P_{X,\min}(t)$ is irreducible over $\Q$.

This implies that if we fix an embedding $E \subset \C$ then $L_X$ is the normal closure of $E$ over $\Q$ and $R_X$ is the set of images of $\Fr_X$ in $\C$ with respect to all field embeddings $E \hookrightarrow \C$; in addition, every eigenvalue $\alpha \in R_X$ has multiplicity $e$. Since $\Fr_X$ generates $E$ (over $\Q$), we conclude that if
$$\phi: E \hookrightarrow \C, \ \psi: E \hookrightarrow \C$$
are two {\sl distinct} field embeddings of $E$ into $\C$ then
$$\phi(\Fr_X) \ne \psi(\Fr_X).$$
Now assume additionally that $X$ is simple and $\dim(X)>1$;  if $\dim(X)=2$ then we also assume  that
$X$ is absolutely simple. Then it is known (\cite[Sect. 1, Exemples]{Tate2}) that $E$ is a CM-field. In particular, it has even degree say, $2d=\deg(\P_{X,\min})$ and for each field embedding $\phi: E \hookrightarrow \C$  its complex-conjugate
$\bar{\phi}: E \hookrightarrow \$$
 does {\sl not} coincide with $\phi$. Now let $\Phi$ be a CM-type of $E$, i.e., a set $\{\phi_1, \dots \phi_d\}$ of $d$ distinct field embeddings $E \hookrightarrow \C$ such that $\phi_j \ne \overline{\phi_j}$ for all  $i,j$. Now if $\bar{\Phi}=\{\bar{\phi}\mid \phi \in \Phi\}$ then $\Phi \bigcap \bar{\Phi}=\emptyset$ and
$$\Phi\cup \bar{\Phi}=\{\phi_1, \dots , \phi_d; \ \overline{\phi_1}, \dots , \overline{\phi_d}\}$$
coincides with the $2d$-element set of all field embeddings $E \hookrightarrow \C$.
It follows that
if we put $\alpha_i=\phi_i(\Fr_X)$ for all $i$ with $1 \le i \le d$ then  $R_X$ consists of $2d$ distinct elements
$$\{\alpha_1, \dots , \alpha_d;\ \overline{\alpha_1}=\frac{q}{\alpha_1}, \dots , \overline{\alpha_d}=\frac{q}{\alpha_d}\}.$$
Notice also that
$$ 2\dim(X)=\deg(\P_X)=e\cdot \deg(\P_{X,\min})=e\cdot 2d.$$
\end{ex}

 We write
$$\Gamma=\Gamma(X,k)$$ for the multiplicative subgroup of $L^{*}$ generated by $R_X$. Using Weil's results mentioned above, one may easily check  that $\Gamma_X$ contains $q$ and is a finitely generated group of rank $\rk(\Gamma)\le \dim(X)+1$. Notice that the rank of $\Gamma$ is $\dim(X)+1$ if and only if $\Gamma$ is a free commutative group of rank $\dim(X)+1$ \cite{ZarhinK3}.

\begin{rem}
\label{nonPP}
 It follows from Remark \ref{nonP} that if $\B$ is a maximal ideal in $\Oc_L$ such that $\fchar(\Oc_L/\B) \ne p$  then all elements of $\Gamma(X.k)$ are $\B$-adic units.
\end{rem}

 We write
$$\Gamma^{\prime}=\Gamma^{\prime}(X,k)$$ for the multiplicative subgroup of $L^{*}$ generated by  all the eigenvalues of $q^{-1}\Fr_X^2$. In other words, $\Gamma^{\prime}$ is the multiplicative (sub)group generated by
$$R_X^{\prime}=\{q^{-1}\alpha^2 \mid \alpha \in R_X\}.$$
Clearly, all the archimedean absolute values of all elements of $\Gamma^{\prime}$ are equal to $1$.

One may easily check that
$$\rk(\Gamma^{\prime})+1=\rk(\Gamma)$$
and $\Gamma^{\prime}$ and $q$ generate a subgroup of finite index in $\Gamma$.
We define the rank of $X$ as $\rk(\Gamma^{\prime})$ and denote it by $\rk(X)$. Clearly,
$$0 \le \rk(X) \le \dim(X).$$
It is known \cite[Sect. 2.9 on p. 277 and Remark 2.9.2 on p. 278]{ZarhinEssen} that if $Y$ is an abelian variety over $k$ then
$$\max(\rk(X),\rk(Y)) \le \rk(X\times Y) \le \rk(X) + \rk(Y).$$
Notice also that $\rk(X)$ does not depend on a field of definition of $X$ and would not change if we replace $X$ by an isogenous abelian variety. In addition, $\rk(X)=0$ if and only if $X$ is a supersingular abelian variety (\cite[Sect. 2.0]{ZarhinEssen}).

This implies the following {\sl trivial} multiplicative relation between eigenvalues $\alpha, \beta, q/\alpha, q/\beta \in R_X$.
$$\alpha \cdot  \frac{q}{\alpha}=q=\beta \cdot  \frac{q}{\beta}. \eqno(3)$$
Let us put
$$R_X^{\prime}:=\{q^{-1}\alpha^2 \mid \alpha \in R_X\}.$$
Clearly, all elements of $R_X^{\prime}$ have archimedean absolute value $1$ with respect to all field embeddings $L \hookrightarrow \C$ and the map $\beta \mapsto \beta^{-1}$ is an involution of $R_X^{\prime}$.

Assume that $k$ is {sufficiently large} with respect to $X$, i.e., the multiplicative group $\Gamma(X,k)$ generated by $k$ does not contain  roots of unity (except $1$). This implies (thanks to Remark \ref{multiplicities}) that all the endomorphisms of $X$ are defined over $k$. On the other hand,
the map
$$R_X \to R_X^{\prime}, \ \alpha \mapsto \alpha^{\prime}=q^{-1}\alpha^2$$
is a bijective map that sends $q/\alpha$ to $1/\alpha^{\prime}$.

 Suppose that there are an
 integer-valued function $e:R_X \to \Z$ and an integer $M$ such that
$$\prod_{\alpha\in R_X} \alpha^{e(\alpha)}=q^M. \eqno(4)$$
Since the archimedean absolute value of each $\alpha$ is $\sqrt{q}$, we have
 $$\frac{1}{2} \left(\sum_{\alpha \in R_X} e(\alpha)\right)=M$$ and therefore
$$2M=\sum_{\alpha \in R_X} e(\alpha), \ \prod_{\alpha\in R_X} \alpha^{2e(\alpha)}=q^{2M}. $$
This implies that        $$\prod_{\alpha\in R_X}(q^{-1} \alpha^2)^{e(\alpha)}=1. \eqno(5)$$
We may rewrite (5) as
$$\prod_{\beta\in R_X^{\prime}} \beta^{e^{\prime}(\beta)}=1 \eqno(\mathrm{5bis})$$
where $e^{\prime}(\alpha^2/q):=e(\alpha)$.

Conversely, if (5bis) holds for some $e^{\prime}:R_X^{\prime} \to \Z$ then we have
$$\prod_{\alpha\in R_X} \alpha^{e(\alpha)}=q^M$$
with $e(\alpha):=2 e^{\prime}(\alpha^2/q)$ and $M:=\sum_{\beta\in R_X^{\prime}}e^{\prime}(\beta)$.
We say that $X$ is {\sl neat} if it enjoys one of the following obviously equivalent conditions (we continue to assume that $k$ is sufficiently large).

\begin{itemize}
\item[(i)]
Suppose an integer-valued function $e: R_X \to \Z$ and a positive integer $M$ satisfy (3). Then $e(\alpha)=e(q/\alpha) \ \forall \alpha \in R_X$.
\item[(ii)]
Suppose an integer-valued function $e^{\prime}: R_X^{\prime} \to \Z$  satisfies (5bis). Then $e^{\prime}(\beta)=e^{\prime}(1/\beta) \ \forall \beta \in R_X^{\prime}$.
\end{itemize}

\begin{rem}
Let us consider the (sub)set
$R_{X,\ssL}$ of $\alpha \in R_X$ such that  $q^{-1}\alpha^2$ is a root of unity. (Here the subscript ss is short for supersingular.) Clearly, $\alpha \in R_{X,\ssL}$ if and only if $q \alpha^{-1} \in R_{X,\ssL}$. It is also clear that if  $R_{X,\ssL}$ is non-empty then $1/2$ is a {\sl slope} of the Newton polygon of $X$ (see Sect. \ref{newton} below). The converse is not true if $\dim(X)>1$.

Recall \cite[Definition 2.3 on p. 276]{ZarhinEssen} that $k$ is {\sl sufficiently large} with respect to $X$ or just sufficiently large if $\Gamma(X,k)$ does not contain roots of unity different from $1$.    If $m$ the order of the subgroup of roots of unity in $\Gamma(X,k)$ and $\kappa/k$  is a finite algebraic field extension then $\kappa$ is sufficiently large for $X$ if and only if the  degree $[\kappa:k]$ is divisible by $m$  \cite[p. 276]{ZarhinEssen}.
In particular, if $k$ is sufficiently large and $\beta \in R_X^{\prime}$ is a root of unity then $\beta=1$. Notice also that if $\rk(X)=\dim(X)$ then $\Gamma(X,k)$ is a free commutative group \cite[Sect. 2.1]{ZarhinK3}, i.e., $k$ is sufficiently large.
\end{rem}

\begin{lem}
\label{super}
Suppose that $k$ is sufficiently large with respect to $X$. If $R_{X,\ssL}$ is non-empty then the following conditions hold:

\begin{itemize}
\item[(i)] $q$ is a square.
\item[(ii)]
$R_{X,\ssL}$ is either the singleton $\{\sqrt{q}\}$ or the singleton  $\{-\sqrt{q}\}$.
In both cases $R_X^{\prime}$ contains
$$q^{-1}(\pm \sqrt{q})^2= 1.$$
\end{itemize}
\end{lem}

\begin{proof}
Let $\alpha \in R_{X,\ssL}$. Since the root of unity $q^{-1}\alpha^2$ lies in
$\Gamma(X,k)$, we conclude that $\alpha^2=q$. Since $R_X$ is $\Gal(L/\Q)$-stable, we conclude that if $q$ is not a square then both $\sqrt{q}$ and $-\sqrt{q}$ lie in $R_X$ and therefore
$$-1=\frac{-\sqrt{q}}{\sqrt{q}}\in \Gamma(X,k),$$
which is not the case, because $k$ is sufficiently large. Therefore $q$ is a square and $R_X$ is either the singleton $\{\sqrt{q}\}$ or the singleton  $\{-\sqrt{q}\}$.
\end{proof}

\begin{rem}
\label{clean}
Suppose that $k$ is sufficiently large. Then if $\alpha_1$ and $\alpha_2$ are {\sl distinct} elements of $R_X$ then
$$\frac{\alpha_1}{\alpha_2} \ne \pm 1$$
and therefore
$q^{-1}\alpha_1^2$ and $q^{-1}\alpha_2^2$
are {\sl distinct} elements of  $R_X^{\prime}$. This implies that
$$\#(R_X)=\#(R_X^{\prime}).$$
\end{rem}

Till the end of this Section we assume that $k$ is sufficiently large with respect to $X$.

In order to compute the rank of {\sl neat} abelian varieties, let us consider the minimal polynomial $\P_{X,\min}(t)$ of the semisimple linear operator
$\Fr_X: V_{\ell}(X)\to V_{\ell}(X)$. The set of roots of
$\P_{X,\min}(t)$ coincides with one of $\P_X(t)$, i.e., with $R_X$; in addition, all the roots of $\P_X(t)$ are simple. It follows from Remark \ref{simple} that if $X$ is simple or $k$-isogenous to a $k$-simple abelian variety  then $\P_{X,\min}(t)$ is irreducible over $\Q$ and  $\P_X(t)=[\P_{X,\min}(t)]^d$ for a certain positive integer $d$.  In general case the minimal polynomial
$$\P_{X,\min}(t)=\prod_{\alpha \in R_X}(t-\alpha)$$
and its degree
$\deg(\P_{X,\min})$ coincides with $\#(R_X)$.

\begin{ex}
\label{supersingular}
 Suppose $X$ a supersingular abelian variety. According to Subsection \ref{NewtonType},  $\alpha^2/q$ is a root of unity for all $\alpha \in R_X$, i.e., $R_X=R_{X,\ssL}$.
It follows from Lemma \ref{super} that $q$ is a square and
  $R_X$ is either the singleton $\{-\sqrt{q}\}$ or the singleton  $\{\sqrt{q}\}$. Then  $\P_{X,\min}(t)$ is a linear polynomial that equals  $t-\sqrt{q}$ or $t+\sqrt{q}$ respectively. This implies that that $\P_X(t)=(t\pm \sqrt{q})^{2\dim(X)}$ and $R_X^{\prime}$ is always the singleton $\{1\}$. It follows that $X$ is neat.
\end{ex}

\begin{ex}
\label{nosuper}
Suppose $R_{X,\ssL}$ is {\sl empty}. This implies that $\alpha \ne q/\alpha$ for every $\alpha \in R_X$, the set $R_X$ consists of even, say, $2d$ elements. Then one may choose $d$ distinct elements $\alpha_1, \dots , \alpha_d$ of $R_X$ such that
$$R_X =\{\alpha_1, \dots, \alpha_d;\  q/\alpha_1, \dots , q/\alpha_d\}.$$
If we put $\beta_i=q^{-1} \alpha_i^2$ then $R_X^{\prime}$ also consists of $2d$ (distinct) elements and coincides with
$$\{\beta_1, \dots, \beta_d; \ \beta_1^{-1}, \dots , \beta_d^{-1}\}.$$
In particular, $\rk(X) \le d$.
Now $X$ is neat if and only if the set $\{\beta_1, \dots, \beta_d\}$ is multiplicatively independent, which means that
$$\rk(X)=d.$$
If this is the case then
$$\rk(X)=d=\frac{\#(R_X)}{2}=\frac{\deg(\P_{X,\min})}{2}.$$
\end{ex}

\begin{ex}
\label{mixed}
Suppose $R_{X,\ssL}$ is non-empty but does {\sl not} coincide with the whole $R_X$. Let us denote by $\alpha_0$ the only element of $R_{X,\ssL}$; as we have seen above, $q$ is a square and $\alpha_0=\pm \sqrt{q}$. This implies  that if $\alpha$ is an element of $R_X$ that is different from $\alpha_0$ then  $\alpha \ne q/\alpha$, the set $R_X\setminus \{\alpha_0\}$ consists of even number of, say, $2d$ elements. Then  one may choose $d$ distinct elements $\alpha_1, \dots , \alpha_d$ of $R_X\setminus \{\alpha_0\}$ such that
$$R_X =\{\alpha_0; \ \alpha_1, \dots, \alpha_d;\  q/\alpha_1, \dots , q/\alpha_d\}.$$
If we put $\beta_i=q^{-1} \alpha_i^2$ then $\beta_0=1$ and $R_X^{\prime}$  consists of $(2d+1)$ distinct elements
$$\{1; \beta_1, \dots, \beta_d; \ \beta_1^{-1}, \dots , \beta_d^{-1}\}.$$
In particular, $\rk(X) \le d$.
Now $X$ is neat if and only if the set $\{\beta_1, \dots, \beta_d\}$ is multiplicatively independent, which means that
$$\rk(X)=d.$$ If this is the case then
$$\rk(X)=d=\frac{\#(R_X)-1}{2}=\frac{\deg(\P_{X,\min})-1}{2}.$$
\end{ex}

\begin{ex}
\label{rank1}
Suppose that $X$ is simple and $\rk(X)=1$. It follows from Lemma 2.10 of \cite{ZarhinEssen} that $R_X^{\prime}$ consists of two elements, say, $\beta$ and $\beta^{-1}$. Clearly, $\beta$ is not a root of unity. This implies easily that $X$ is neat.
\end{ex}

We will need the following elementary lemma.

\begin{lem}
\label{elementary}
Let $p$ be a prime, $B$ an imaginary quadratic field, $T$ the set of maximal ideals in $B$ that lie above $p$. Let $U_T \subset B^{*}$ be the multiplicative subgroup of $T$-units in $B$ and $U_T^{1}$ the subgroup of $T_B$ that consists of all $\gamma \in U_T$ such that  the archimedian absolute value of $\gamma$ is  $1$.  If $U_T^{1}$ is infinite then $p$ splits in $B$ (i.e., $\#(T)=2$), $\rk(U_T)=2$ and $\rk(U_T^{1})=1$.
\end{lem}

\begin{proof}
By the generalized  Dirichlet's unit theorem \cite[Ch. V, Sect. 1]{LangA}, $U_T$ is a finitely generated commutative group of rank $\#(T)$.
Clearly, $U_T$  contains the element $p$ of infinite order.  If $\#(T)=1$ then $\rk(U_T)=1$ and therefore for each $$\gamma \in  U_T^{1} \subset U_T$$ a certain positive power of $\gamma$ is a power of $p$. However,  the archimedean absolute value of $\gamma$ equals $1$ and therefore $\gamma$ must be a root of unity, which is not the case, since there are only finitely many roots of unity in $B$. So, $\#(T)=2$, i.e., $p$ splits in $B$. In addition, $U_T$ has rank $2$. Since no power of $p$ (except $1=p^0$)  lies in $U_T^{1}$, we conclude that  $\rk(U_T^{1})< \rk(U_T)=2$. Since $\rk(U_T^{1})\ge 1$, we conclude that $\rk(U_T^{1})= 1$.
\end{proof}

\begin{cor}
\label{rankquad}
Let $B$ be an imaginary quadratic subfield in $L$. Suppose that the intersection
$$\Gamma^{\prime}(X,k)_B:=B \bigcap  \Gamma^{\prime}(X,k)$$
of $B$ and $\Gamma^{\prime}(X,k)$ is infinite. Then $p$ splits in $B$ and the infinite multiplicative group $\Gamma^{\prime}(X,k)_B$ has rank $1$.
\end{cor}

\begin{proof}
Notice that (in the notation of Lemma \ref{elementary}) $\Gamma^{\prime}(X,k)_B$ is an infinite subgroup of $U_T^{1}$. In particular, $U_T^{1}$ is also infinite. Now Corollary follows readily from Lemma \ref{elementary}.
\end{proof}

\begin{thm}
\label{inert}
Suppose $X$ is simple, $\dim(X)>1$, $k$ is sufficiently large with respect to $X$ and the degree $[E:\Q]$
of the CM-field $E=\Q[\Fr_X]$ is an even number $2d$ that is strictly greater than $2$.
Suppose that $E$ contains an imaginary quadratic field $B$ such that $p$ does not split in $B$. Then
$$\Norm_{E/B}(q^{-1}\Fr_X^2)=1\in B$$ and $\rk(X)<d$.
\end{thm}

\begin{proof}
By Remark \ref{multiplicities} $X$ is absolutely simple.
Let us fix an embedding of $B$ into $\C$ and view $B$ as the subfield of $\C$. Let $\Phi$ be the $d$-element set of field embeddings $\phi_i: E \hookrightarrow \C$ that coincide on $B$ with the identity map ($1\le i \le d$). Then $\Phi$ is a
CM-type of $E$ and, thanks to Example \ref{simpleCM}
$R_X$ consists of $2d$ distinct elements
$$\{\alpha_1, \dots , \alpha_d;\ \frac{q}{\alpha_1}, \dots , \frac{q}{\alpha_d}\}$$
where
$\alpha_i=\phi_i(\Fr_X)$ ($1\le i \le d$).
Clearly, (in the notation of Corollary \ref{rankquad})
$$\gamma:=\Norm_{E/B}(q^{-1}\Fr_X^2) =\prod_{i=1}^d \frac{\alpha_i^2}{q} \in B \bigcap  \Gamma^{\prime}(X,k)=
\Gamma^{\prime}(X,k)_B.$$
By Corollary \ref{rankquad} the group $\Gamma^{\prime}(X,k)_B$ is finite and therefore $\gamma$ is a root of unity. Since $k$ is sufficiently large, $\gamma=1$, i.e.
$$\Norm_{E/B}(q^{-1}\Fr_X^2) =\prod_{i=1}^d \frac{\alpha_i^2}{q}=1.$$
This proves the first assertion about the norm. Clearly, $\Gamma^{\prime}(X,k)$ is generated by $d$ elements
$\{q^{-1} \alpha_i^2 \mid 1 \le d\}$ and their product is $1$. This implies that $\Gamma^{\prime}(X,k)$ is actually  generated by the first $(d-1)$ elements $\{q^{-1} \alpha_i^2 \mid 1 \le d-1\}$
 and we are done.
\end{proof}

\section{Ranks of non-simple abelian varieties}
\label{nonsimple}

The following assertion was proven in \cite[pp. 273, 280--281]{ZarhinEssen}.

\begin{thm}
\label{essen212}
Let $X$ and $Y$ be non-supersingular simple abelian varieties over $k$. If
$$\rk(X\times Y)=\rk(X)+\rk(Y)-1$$
then there exists an imaginary quadratic field $B$ enjoying the following properties.

\begin{itemize}
\item[0)]
$p$ splits in $B$;
\item[1)]
The number fields $E_X=\Q[\Fr_{X,k}]$ and $E_Y=\Q[\Fr_{Y,k}]$ contain subfields isomorphic to $B$;

\item[2)]
$\Norm_{E_X/B}(q^{-1}\Fr_{X,k}^2)$ and $\Norm_{E_Y/B}(q^{-1}\Fr_{Y,k}^2)$ are not roots of unity.
\end{itemize}
\end{thm}

\begin{rem}
There is a typo in the displayed formula for ranks in \cite[Th. 2.12]{ZarhinEssen}, see Sect. \ref{corr}. It was also erroneously claimed (without a proof) in \cite[Th. 2.12]{ZarhinEssen} that the conditions 0,1,2 are equivalent to the formula $\rk(X\times Y)=\rk(X)+\rk(Y)-1$. Actually, the conditions 0,1,2  imply only the inequality $\rk(X\times Y)\le \rk(X)+\rk(Y)-1$.
\end{rem}

\begin{proof}[Proof of Theorem \ref{essen212}]
Assertions 1 and 2 are proven in \cite[pp. 280--281]{ZarhinEssen}. Assertion 0 is proven in \cite[Remark 1.1.5 on p. 273]{ZarhinEssen}. (It also follows from Assertion 2 combined with Lemma \ref{rankquad}).

\end{proof}

\begin{cor}[Theorem 2.11 of \cite{ZarhinEssen}]
\label{essen211}
Assume that $E=\End^0(X)$ is a number field. Let $Y$ be an ordinary elliptic curve over $k$. The equality $\rk(\Gamma(X\times Y))=\rk(\Gamma(X))$ holds true if and only if $\End^0 X$ contains an imaginary quadratic subfield isomorphic to $B=\End^0 Y$ and $\Norm_{E/B}(q^{-1} \Fr_{X,k}^2)$ is not a root of unity.
\end{cor}

\begin{proof}
Since $\rk(Y)=1$, we have
$$\rk(X)=\rk(X)+\rk(Y)-1.$$
This implies that in one direction (if we are given that $\rk(\Gamma(X\times Y))=\rk(\Gamma(X))$, i.e., $\rk(X\times Y)=\rk(X)$) then our assertion follows from Theorem \ref{essen212}.  Conversely, suppose that $B=\End^0 Y$ is isomorphic to a subfield of $E$ and $$\gamma:=\Norm_{E/B}(q^{-1} \Fr_{X,k}^2)\in B$$ is {\sl not} a root of unity. Let us fix an embedding $E\subset\C$. We have
$$\gamma \in B \subset E \subset L_X\subset \C.$$
 By definition, $\gamma$ is a product of elements of $R_X^{\prime}$ and therefore lies in $\Gamma^{\prime}(X,k)$. In particular,
in the notation of Lemma \ref{rankquad}, $\gamma \in \Gamma^{\prime}(X,k)_B$.
On the other hand, $q^{-1}\Fr_{Y,k}^2\subset B$ is also {\sl not} a root of unity; in addition,  it generates $\Gamma^{\prime}(Y,k)$.  Notice that (in the notation of Lemma \ref{elementary}) both $\gamma$ and $q^{-1}\Fr^2_{Y,k}$ lie in $U_T^{1}$; in particular, $U_T^{1}$ is infinite. By Lemma \ref{elementary}, $U_T^{1}$  has rank $1$ and therefore the intersection  of two cyclic (sub)groups generated by $\gamma$ and $q^{-1}\Fr^2_{Y,k}$ respectively is an infinite cyclic group. This implies that the intersection of finitely generated groups $\Gamma^{\prime}(X,k)$ and $\Gamma^{\prime}(Y,k)$ is an infinite group. It follows that the rank of  $\Gamma^{\prime}(X\times Y,k)=\Gamma^{\prime}(X,k)\Gamma^{\prime}(Y,k)$ is strictly less than the sum
$$\rk(\Gamma^{\prime}(X,k))+\rk(\Gamma^{\prime}(Y,k))=\rk(\Gamma^{\prime}(X,k))+1.$$
In other words, $\rk(X\times Y)< \rk(X)+1$, i.e., $\rk(X\times Y)\le \rk(X)$. It follows that $\rk(X\times Y)=\rk(X)$ and we are done.
\end{proof}

\section{Newton polygons}
\label{newton}
In order to define the Newton polygon of $X$, let us consider the ring $\Oc_L$ of integers in $L$ and pick a maximal ideal $\PP$ in $\Oc_L$ such that the residue field $\Oc_L/\PP$ has characteristic $p$. The set $S_p$ of such ideals constitutes a $\Gal(L/\Q)$-orbit. Let
$$\ord_{\PP}: L^{*} \to \Q$$
be the discrete valuation map that corresponds to $\PP$ and normalized by the condition
$$\ord_{\PP}(q)=1.$$
Then the set
$$\Sl_X=\ord_{\PP}(\RR_X) \subset \Q$$
is called the {\sl set of slopes} of $X$. For each $c\in \Sl_X$ we write
$$\length(c)=\length_X(c)$$
for the number of roots $\alpha$ of $\P_X(t)$ (with multiplicities) such that
$$\ord_{\PP}(\alpha)=c.$$
By definition
$$\sum_{c \in \Sl_X} \length(c)=\deg(\P_X)=2\dim(X). \eqno(6)$$

\begin{rem}
\label{fraction}
It is well known that all slopes $c\in \Sl_X$ are rational numbers that lie between $0$ and $1$.  In addition, if $c$ is a slope then $1-c$ is also a slope and $\length(c)=\length(1-c)$. In addition, if $1/2$ is a slope then its length is even.
Notice also that the rational number $c$ can be presented as a fraction, whose denominator is a positive integer that does not exceed $2\dim(X)$ \cite[p. 173]{ZarhinInv79}.

Since $\P(t)$ has rational coefficients and $\Gal(L/\Q)$ acts transitively on $S_p$, the set $\Sl_X$ and the function
$$\length_X: \Sl_p \to \NN$$
do not depend on a choice of $\PP$. The {\sl integrality property} of the Newton polygon \cite[Sect. 9 and 21]{OortG} means that $c \cdot \length_X(c)$ is a positive integer for each nonzero slope $c$. Suppose that a slope $c \ne 1/2$ is presented as the fraction in lowest terms, whose denominator is greater than $\dim(X)$. Then $\length(c)>\dim(X)$
and
$$\length(1-c)=\length(c)>\dim(X),$$
which implies $\length(c)+\length(1-c)>2\dim(X)$. This contradicts  to (6). So, each slope $c \ne 1/2$ can be presented as a fraction, whose denominator does not exceed $\dim(X)$. It is also clear, that if the denominator of $c$ in lowest terms is exactly $\dim(X)$ then
$$\length(c)=\dim(X)=\length(1-c)$$
and $\Sl_X=\{c,1-c\}$.
\end{rem}

\begin{rem}
\label{Slopes}
Suppose that $X$ is simple. Then as we have seen, $\P_X(t)=\P_{X, \min}(t)^e$. It follows that $e$ divides $\length_X(c)$ for every slope $c$ of the Newton polygon of $X$.
\end{rem}

\begin{defn}
\label{NewtonType}
An abelian variety $X$ is called {\sl ordinary} if $\Sl_X=\{0,1\}$; it is called {\sl supersingular} if $\Sl_X=\{1/2\}$. It is well known that $X$ is supersingular if and only if $R_X^{\prime}$ consists of roots of unity, i.e., $q^{-1}\alpha^2$ is a root of unity for all $\alpha\in R_X$.
 (By the way, it follows immediately from Proposition 3.1.5 in \cite[p. 172]{ZarhinInv79}.)

 $X$ is called of {\sl K3 type} \cite{ZarhinK3} if
$\Sl_X$ is either $\{0,1/2, 1\}$ or $\{0,1\}$ while (in both cases) $\length_X(0)=\length_X(1)=1$.
It is called {\sl almost ordinary} \cite{LenstraZarhin} if
$$\Sl_X=\{0,1/2, 1\}, \ \length_X(1/2)=2.$$
\end{defn}

\begin{rem}
\label{bigrank}
Clearly, $X$ is supersingular if and only if $\rk(X)=0$. If $X$ is a simple abelian variety of K3 type then
$\End^0(X)$ is a field and $\rk(X)=\dim(X)$ \cite{ZarhinK3}.
It is  known (\cite[Th. 7.2  on p. 553]{W}) that if $X$ is a simple ordinary abelian variety then  $\End^0(X)$ is a field
and all endomorphisms of $X$ are defined over $k$. In particular,
  $X$ is absolutely simple.
If $X$ is a simple almost ordinary then $\End^0(X)$ is a field \cite{OortLift}. It is also known that for such $X$ we have
$\rk(X)=\dim(X)$ or $\dim(X)-1$; if, in addition, $\dim(X)$ is {\sl even} then $\rk(X)=\dim(X)$ \cite{LenstraZarhin}.
\end{rem}

\begin{thm}
\label{sloperank}
Let $X$ be a simple abelian variety of positive dimension over $k$ and suppose that $\P_X(t)$ is irreducible. Suppose that there exists a rational number $c \ne 1/2$ such that $\Sl_X=\{c, 1-c\}$. (E.g., $X$ is ordinary.)  If $\rk(X)=\dim(X)-1$ then $\dim(X)$ is even.
\end{thm}

\begin{proof} Let us put $g=\dim(X)$ and
$$c^{\prime}=2c-1=-[2(1-c)-1].$$
Clearly, $c^{\prime}\ne 0$ and for all $\alpha\in R_X$ the rational number
$\ord_{\PP}(\alpha^2/q)$ is either $c^{\prime}$ or $-c^{\prime}$ . Let us define $m(\alpha)$ by
$$\ord_{\PP}(\alpha^2/q)=m(\alpha) c^{\prime}.$$
Clearly, $m(\alpha)=1$ or $-1$.
By Theorem 3.6(b) of \cite{LenstraZarhin} there  exist $\alpha_1, \dots , \alpha_g \in R_X$ and integers $n_1, \dots , n_g$ such that every $n_i$ is either $1$ or $-1$ and
$\gamma=\prod_{i=1}^g (\alpha_i^2/q)^{n_i}$ is a root of unity. Pick $\PP \in S_p$. We have
$$0=\ord_{\PP}(\gamma)=\sum_{i=1}^g n_i \ord_{\PP}\left(\alpha_i^2/q\right)=\sum_{i=1}^g n_i m(\alpha_i)  c^{\prime}=\left[\sum_{i=1}^g (\pm 1)\right] c^{\prime}.$$
It follows that for a certain choice of signs $\sum_{i=1}^g (\pm 1)=0$ and therefore $g$ is even.
\end{proof}

\begin{cor}
\label{rightNP}
Suppose that $X$ is a simple abelian variety over $k$. Assume that $1 \le \dim(X)\le 3$ and $k$ is sufficiently large with respect to $X$. If $X$ is not neat then it is almost ordinary and $\dim(X)=3$.
\end{cor}

\begin{proof}
It follows from Remark \ref{multiplicities} that $X$ is absolutely simple.
The equality $\dim(X)=3$ follows from Theorem 3.5 in \cite{ZarhinEssen}. Since $X$ is not neat,
$1< \rk(X)< \dim(X)=3$. This implies that
$\rk(X)=2$ and therefore $\deg(\P_{X,\min}) > 2 \cdot 2=4$.
Since $\deg(\P_{X,\min})$ divides $\deg(\P_X)=6$, we conclude that
 $\deg(\P_{X,\min})=\deg(\P_X)$, i.e., $\P_X(t)=P_{X,\min}(t)$ is irreducible over $\Q$. Since $\dim(X)=3$ is {\sl odd}, it follows from Theorem \ref{sloperank} that the Newton polygon of $X$ has, at least, 3 distinct slopes. In addition,
 Remark \ref{fraction}  implies that  all the slopes different from $1/2$ can be presented as fractions, whose denominator is strictly less than $\dim(X)$. In other words, $\Sl_X = \{0, 1/2,1\}$. In particular, $\length(1/2)=2$ or $4$. If $\length(1/2)=4$ then $\length(0)=\length(1)=1$ and  $X$ is  of K3 type, which is not the case,  since the rank of a simple abelian variety of K3 type equals its dimension \cite{ZarhinK3}. Therefore $\length(1/2)=2$ and $\length(0)=\length(1)=2$, i.e., $X$ is almost ordinary.
\end{proof}

\section{Abelian Surfaces}
\label{dim2}

The following statement may be viewed as a rewording of Example (2)(A) in \cite[Sect. 8.4, p. 64]{ShimuraCM}.
(See \cite[Th. 5]{Shimura}, \cite{OortEndo} and also \cite{OZ} where the endomorphism algebras of abelian varieties and complex tori are discussed in detail.)

\begin{thm}
\label{surfaceIM}
Let $\L$ be a quartic CM-field that contains an imaginary quadratic field $B$. Let $S$ be a complex abelian surface provided with an embedding $\L \hookrightarrow \End^0(S)$. Then $S$ is isogenous to a square of an elliptic curve with complex multiplication. In particular, $S$ is not simple.
\end{thm}

\begin{proof}
We may view $\L$ as a subfield of $\C$. Then $B=\Q(\sqrt{-d})$  where $d$ is a positive integer. The field $\L$ contains the real quadratic subfield $\Q(\sqrt{r})$ where $r$ is a square-free positive integer. Clearly,
$$\L=\Q\oplus \Q\sqrt{-d}\oplus \Q\sqrt{r}\oplus\Q\sqrt{-rd}$$
is a Galois extension of $\Q$.
 This implies that $\L$ contains a second imaginary quadratic subfield $H:=\Q(\sqrt{-rd})$. The natural map
$B\otimes_{\Q} H \to \L, \ b\otimes h \mapsto bh$
 is a field isomorphism. In addition,
the natural injective homomorphism
$$\Gal(B/\Q) \times \Gal(H/\Q) \hookrightarrow \Gal(\L/\Q)$$
is surjective and therefore is a group isomorphism.  Since $[\L:\Q]=2\cdot 2$, it admits $2^2=4$ CM-types $\Phi$ \cite[Sect. 22]{Mumford}, \cite{Lang}. Here is the list of all them. We have two CM-types $\Gal(B/\Q)\otimes \tau_2$ indexed by $\tau_2 \in \Gal(H/\Q)$ and two CM-types $\tau_1 \otimes \Gal(H/\Q)$ indexed by $\tau_1 \in \Gal(B/\Q)$.  They all have {\sl nontrivial} automorphism groups
$$\Aut(\Phi):=\{\sigma \in \Gal(\L/\Q)\mid \sigma\Phi=\Phi\}.$$
Namely, $\Aut(\Phi)=\Gal(B/\Q)$ for the former two CM-types and $\Aut(\Phi)=\Gal(H/\Q)$
for the latter two. Now the result follows from Theorem 3.5 of \cite[p. 13]{Lang}
(applied to $F=\L$.)
\end{proof}

\begin{cor}
\label{nonexist}
There does not exist an abelian surface $Y$ over a finite field $k$ that enjoys the following properties.

\begin{itemize}
\item[(i)]
All endomorphisms of $Y$ are defined over $k$.
\item[(ii)]
$\End^0(Y)$ is a quartic CM-field that contains an imaginary quadratic subfield.
\end{itemize}
\end{cor}

\begin{proof}
Assume that such $Y$ does exist. Then it is absolutely simple.
Replacing if necessary, $k$ by its finite overfield and $Y$ by a $k$-isogenous abelian variety, we may and will assume that $Y$ can be {\sl lifted} to an abelian variety $A$ in characteristic zero such that there is an embedding $\End^0(Y)\hookrightarrow \End^0(A)$ \cite[Sect. 3, Th. 2]{Tate2}. It follows that $A$ is absolutely simple, which contradicts Theorem \ref{surfaceIM}. The obtained contradiction proves Corollary.
\end{proof}

\section{Proof of Theorem \ref{main}}
\label{mainproof}
Assume that $X$ is {\sl not} neat, $k$ is sufficiently large  and $1 \le \dim(X) \le 3$. According to \cite[Th. 3.5 on p. 283]{ZarhinEssen}, $\dim(X)=3$ and one of the following two conditions holds.

\begin{itemize}
\item[(a)] $X$ is simple, $E=\End^0(X)$ is a number field that contains an imaginary quadratic subfield $B$ such that $\Norm_{E/B}(q^{-1}\Fr_X^2)$ is  a root of unity.
\item[(b)]
$X$ is isogenous over $k$ to a product $Y \times Z$ of a {\sl simple} abelian surface $Y$ and an elliptic curve $Z$; $\End^0(Y)$ is a quartic CM-field containing an imaginary quadratic subfield.
\end{itemize}

It follows from Corollary \ref{nonexist} that such an $Y$ does not exist.  Indeed, $\Gamma(X,k)=\Gamma(Y,k)\Gamma(Z,k)$; in particular, $\Gamma(Y,k)$ does not contain nontrivial roots of unity. Therefore all endomorphisms of $Y$ are defined over $k$. Hence $Y$ is absolutely simple. This contradicts to  Corollary \ref{nonexist} and implies that the case (b) does not occur.

In the case (a), Corollary \ref{rightNP}
implies that $X$ is  almost ordinary.  Let us fix a field embedding $B \subset \C$ and let
$$\sigma_1, \sigma_2,\sigma_3: E \hookrightarrow \C$$ be the list of field embedding $E \to \C$ that coincide with the identity map on $B$. Clearly, $\{\sigma_1,\sigma_2,\sigma_3\}$ is a CM-type of $E$.
Let us put
$$\alpha_1=\sigma_1(\Fr_X)\in \C, \ \alpha_2=\sigma_2(\Fr_X)\in \C, \alpha_3=\sigma_3(\alpha_3)\in \C.$$
It follows from Example \ref{simpleCM} that
$$R_X=\{\alpha_1,\alpha_2,\alpha_3; \ q/\alpha_1,q/\alpha_2,q/\alpha_3\}.$$
This implies that
$$L=\Q(R_X)=\Q(\alpha_1,\alpha_2,\alpha_3)=B(\alpha_1,\alpha_2,\alpha_3)$$
and the root of unity
$$\Norm_{E/B}\left(q^{-1}\Fr_X^2\right)=\prod_{i=1}^3 \sigma_i \left(q^{-1} \Fr_X^2\right)=q^{-3}\prod_{i=1}^3  \alpha_i^2 \in \Gamma(X,k).$$
 Since $\Gamma(X,k)$ does {\sl not} contain nontrivial roots of unity,
 $$\Norm_{E/B}(q^{-1}\Fr_X^2)=1.$$
 (By the way, this gives us the relation
 $$q^3=\left(\prod_{i=1}^3 \alpha_i\right)^2.)$$
 This ends the proof.

\section{Examples}
\label{exam}
Throughout this section, $p$ is a prime, $B$ an imaginary quadratic field such that $p$ does {\sl not} split in $B$, i.e., the tensor product
$$B_p:=B \otimes_{\Q}\Q_p$$
is a field that is a quadratic extension of $\Q_p$. We fix an embedding of $B$ into $\C$ and view $B$ as a certain subfield of $\C$. For the sake of simplicity,  let us assume that $p$ is {\sl unramified} in $B$.

 Let $K$ be a totally real  cubic field such that the tensor product
$$K_p:=K \otimes_{\Q}\Q_p$$
is isomorphic to a direct sum $\Q_p\oplus B_p $.  Since $\Q_p$ and $B_p$ are non-isomorphic field extensions of  $\Q_p$, the cubic field extension  $K/\Q$ is not Galois. Let $\Oc_K$ be the ring of integers in the number field $K$. The description of $K_p$ means that the ideal $p\Oc_K$ splits into a product $\B_1 \B_2$ of two distinct  maximal ideals
$\B_1$ and $\B_2$ such that the completion $K_{\B_1}$ with respect to $\B_1$-topology is $\Q_p$ while the
completion $K_{\B_2}$ with respect to $\B_2$-topology is isomorphic to $B_p$.

Let $\tilde{K}$ be the normal closure of $K$. Clearly, $\tilde{K}/K$ is a quadratic extension and the Galois group of
$\tilde{K}/\Q$ is the full symmetric group $\ST_3$. Let $K_2$ be the subfield of $\A_3$-invariants in $\tilde{K}$ where $\A_3$ is the corresponding alternating (sub)group.  Then $K_2/\Q$ is a quadratic field extension that coincides with the maximal abelian subextension of $\tilde{K}/\Q$.

Clearly, $K$ and $B$ are linearly disjoint over $\Q$. Let us consider its tensor product (compositum)
$$E:=B\otimes_{\Q}K;$$
it is a sextic CM-field, which is an imaginary quadratic extension of totally real $K=1\otimes K$ with the {\sl complex conjugation}
$$x\otimes y \mapsto \bar{x}\otimes y \ \forall x \in B\subset \C, \ y \in K.$$
Similarly, $\tilde{K}$ and $B$ are linearly disjoint over $\Q$ and its  tensor product (compositum)
$L:=B\otimes_{\Q}\tilde{K}$  is a degree 12  Galois closure of $E$. (Actually, $L$ is a quadratic extension of $E$ It follows from \cite[Lemma 18.2.iii]{ShimuraCM} that $L$ is a CM-field.)  In addition, the natural group homomorphism
$$ \Gal(B/\Q) \times \Gal(\tilde{K}/\Q) \to \Gal(L/\Q)$$
is an isomorphism. In particular, the quartic field extension $(B\otimes_{\Q}K_2)/\Q$ is the maximal abelian subextension of $L/\Q$. In addition,
$$\Gal((B\otimes_{\Q}K_2)/\Q)=\Gal(B/\Q) \times \Gal(K_2/\Q)$$
is a product of two cyclic groups of order 2. Since every root of unity in $L$ must lie in the (quartic) maximal abelian $B\otimes_{\Q}K_2$, the number $m_L$ of roots of unity in $L$ is $2,4,6$ or $8$. (It cannot be $10$, since the order 4 Galois group of the fifth cyclotomic field over $\Q$ is cyclic.)

Notice that every proper {\sl maximal subfield} $F$ of $E$ is either $B=B\otimes 1$ or $K=1\otimes K$. Indeed, $F$ is either quadratic or cubic extension of $\Q$. If $F$ is quadratic and does not coincide with $B$ then $F$ and $B$ are linearly disjoint over $\Q$, their tensor product $B\otimes F$ is a quartic field extension and the natural map
$$B\otimes F \to E,  \ x,y \mapsto xy$$
is a field embedding. This implies that sextic $E$ contains a quartic subfield, which is not the case, since $4$ does {\sl not} divide $6$. Now assume that $F$ is cubic. Since $F$ is a subfield of the CM-field $E$,  it follows from  \cite[Lemma 18.2.iv]{ShimuraCM} that $F$ is either totally real or a CM-field. Since $3=[F:\Q]$ is odd, $F$ is not a CM-field. It follows that $F$ is totally real and therefore lies in $K$. This implies that $F=K$.

Let us fix an embedding of $L$ into $\C$ that acts as the identity map on $B$. Notice that there is a canonical isomorphism of semisimple $\Q$-algebras
$$B \otimes_{\Q}B \cong B\oplus B, \ u\otimes v \mapsto (uv, \bar{u}v).  \eqno(7);$$
The complex conjugation on the first factor $B$ (on the left hand side) permutes the summands $B$'s (on the right hand side).
Tensoring this isomorphism by $\Q_p$ (over $\Q$), we obtain the canonical isomorphism of semisimple $\Q_p$-algebras
$$B_p \otimes_{\Q_p}B_p \cong B_p\oplus B_p  \eqno(7\mathrm{bis})$$
such that the nontrivial automorphism of $B_p/\Q_p$ (that acts on the first factor $B_p$ on the left hand side) permutes the summands (on the right hand side).
We have
$$E_p:=E \otimes_{\Q}\Q_p=[B\otimes_{\Q}K]\otimes_{\Q}\Q_p=[B \otimes_{\Q}\Q_p]\otimes_{\Q_p}[K \otimes_{\Q}\Q_p]=
B_p\otimes_{\Q_p}K_p=$$
$$B_p\otimes_{\Q_p}[\Q_p\oplus B_p]=B_p\oplus [B_p\otimes_{\Q_p}B_p].$$
By (7bis) the second summand in the right hand side is a direct sum of two copies of  $B_p$. In addition
 the complex conjugation on $E=B\otimes_{\Q}K$  permutes these two copies. On the other hand, the conjugation leaves invariant the first summand $B_p$, acting on it as the only nontrivial automorphism of $B_p/\Q_p$ (induced by the complex conjugation on $B$). This implies that if $\Oc_E$ is the ring of integers in $E$ then the ideal $p\Oc_E$ splits into a product of three {\sl distinct} maximal ideals $\PP_0\PP_1\PP_2$ and these ideals enjoy the following properties.
\begin{itemize}
\item
$\PP_2$ coincides the complex-conjugate $\overline{\PP_1}$  of $\PP_1$.
\item
 $\PP_0$ coincides with its own complex-conjugate  $\overline{\PP_0}$.
\item
Each completion $E_{\PP_i}$ of $E$ in $\PP_i$-adic topology is isomorphic to $B_p$. In particular,
$$[E_{\PP_i}:\Q_p]=[B_p:\Q_p]=2 \  \forall i=0,1,2.$$
\end{itemize}

Let us consider the (nonzero) ideal $\PP=\PP_0\PP_1^2$ in $\Oc_E$. Clearly, its complex-conjugate $\bar{\PP}$ coincides
with $\PP_0\PP_2^2$. We have
$$\PP \bar{\PP}=(\PP_0\PP_1^2)(\PP_0\PP_2^2)=(\PP_0\PP_1\PP_2)^2=p^2\Oc_E.$$
There exists a positive integer $h$ such that $\PP^h$ is a principal ideal, i.e., there exists a nonzero $\beta\in\Oc_E$ such that
$$\PP^h=\beta \cdot \Oc_E.$$
It follows that its complex-conjugate
$\overline{\PP^h}=\bar{\beta} \cdot \Oc_E$
and
$$p^{2h}\cdot \Oc_E=(\PP \bar{\PP})^{2h}=\PP^h \overline{\PP^h}=\beta\bar{\beta} \cdot \Oc_E.$$
This implies that the ratio
$$u=\frac{p^{2h}}{\beta\bar{\beta}}$$
is a unit in $\Oc_E$. Clearly, $\bar{u}=u$, i.e., $u \in \Oc_K^{*}$. Let us put
$$\pi:=u\beta^2\in \Oc_E.$$
We have
$$\pi\bar{\pi}=p^{4h}, \ \pi\cdot \Oc_E=(\PP^h)^2=\PP^{2h}=\PP_0^{2h}\PP_1^{4h}, \
\bar{\pi}\cdot \Oc_E=\PP_0^{2h}\PP_2^{4h}.$$
 I claim that $\Q[\pi]$ coincides with $E$.
Indeed, if $\Q[\pi]$ does not coincides with $E$ then either $\pi \in B$ or $\pi \in K$.
Suppose that $\pi \in B$. Then $p^{2h}/\pi$ is a $p$-unit in $B$ with archimedean absolute value $1$.
It follows from Lemma \ref{elementary} that  $p^{2h}/\pi$ is  a root of unity. This implies that
$$\PP_0^{2h}\PP_1^{4h}=\pi \Oc_E=p^{2h}\Oc_E=(\PP_0\PP_1\PP_2)^{2h}$$
and therefore $\PP_1^{2h}=\PP_2^{2h}$,
which is not the case. This implies that $\pi$ does {\sl not} belong to $B$.
Suppose that  $\pi \in K$. Then
$$\PP_0^{2h}\PP_1^{4h}=\pi \Oc_E=\bar{\pi}\cdot \Oc_E=\PP_0^{2h}\PP_2^{4h}$$
and  therefore $\PP_1^{4h}=\PP_2^{4h}$,
which is not the case. This implies that $\pi$ does {\sl not} belong to $K$.

The same arguments work if we replace $h$ by $nh$ and $\pi$ by $\pi^n$ (for any positive integer $n$). This implies
 that $\Q[\pi^n]=E$ for all $n$.

Let us put
$$q:=p^{4h}.$$
Then $\pi$ is a Weil $q$-number in a sense of Honda--Tate \cite{Tate2}. By Honda-Tate theory \cite{Tate2} there exists a simple abelian variety $X$ over $\F_q$ such that $R_X$ coincides with the Galois orbit of $\pi$. We also have a canonical field isomorphism $\Q[\Fr_X]\cong \Q[\pi]=E$ that sends $\Fr_X$ to $\pi$.  Recall that $E$ is a CM-field; in particular, it has no {\sl real} places. By a theorem of Tate, $\End^0(X)$ is a finite-dimensional central division algebra over $E$, whose local invariants are zero outside divisors of $p$ while for each divisor $\PP_i$ of $p$ the local invariant of $\End^0(X)$  over $E_{\PP_i}$ is
$$c_{\PP_i}:=[E_{\PP_i}:\Q_p]\cdot  \frac{\ord_{\PP_i}(\pi)}{\ord_{\PP_i}(q)}\bmod \Z$$
where
$$\ord_{\PP_i}: E^{*} \to \Z$$
is the discrete valuation map attached to $\PP_i$ (\cite[Sect. 1, Th. 1]{Tate2}; see also \cite[Th. 5.4]{OortG}).  We have
$$c_{\PP_1}=2\cdot \frac{4h}{4h}\bmod \Z=2 \bmod \Z= 0, \ c_{\PP_2}=2\cdot \frac{0}{4h}\bmod \Z=0,$$ $$c_{\PP_0}=2\cdot \frac{2h}{4h}\bmod \Z=1 \bmod \Z=0.$$
We obtain that all the local invariants of $\End^0(X)$  are zero,  i.e.,  $\End^0(X)$ coincides with its center $E$.  In addition, $\dim(X)=[E:\Q]/2=3$. Since $E=\Q[\pi)=\Q[\pi^n]$ for all positive integers $n$, all endomorphisms of $X$ are defined over $\F_q$ \cite[Prop. 5.11]{OortG}. In particular, $X$ is absolutely simple.
Let us consider the $3$-element set $\Phi$ of all field embeddings
$$\sigma_i: E \hookrightarrow \C, \ i=1,2,3$$ that coincide with the identity map on $B$. As above, $R_X$ consists of six distinct elements
$$\{\alpha_1, \alpha_2, \alpha_3; q/\alpha_1, q/\alpha_2, q/\alpha_3\}$$
where $\alpha_i=\sigma_i(\pi)$ for all $i$.

 Now let $k/\F_q$ be a  degree $m_L$ field extension. Let $X_k =X\times_{\F_q}k$ be the abelian threefold obtained by the base change. Clearly, $k$ is sufficiently large with respect to $X$. It is also clear that there is an isomorphism $\End^0(X_k)\cong E$ that sends $\Fr_{X_k}$  to $\pi^{m_L}$; in addition, $R_{X_k}$ consists of six distinct elements
$$\{\gamma_1=\alpha_1^{m_L}, \gamma_2=\alpha_2^{m_L}, \gamma_3=\alpha_3^{m_L}; \ q^{m_L}/\gamma_1, q^{m_L}/\gamma_2, q^{m_L}/\gamma_3\}.$$

\begin{thm}
The abelian variety $X_k$ over $k$ enjoys the properties i)--iii) of Theorem \ref{main}.
\end{thm}

\begin{proof}
By Theorem \ref{inert}
$$1=\Norm_{E/B}(\pi^n)=\prod_{i=1}^3 \left(q^{-n}\gamma_i^2\right).$$
It follows that the abelian threefold $X_k$ is {\sl not} neat.
Now the result follows from  Theorem \ref{main}.
\end{proof}

Now let us construct explicitly  $B$ and $K$ as above. For the sake of simplicity. let us assume that $p$ is odd. Fix a positive integer $d$ such that $-d \bmod p$ is {\sl not} a square in $\F_p$. Let us put $B=\Q(\sqrt{-d})$. Clearly, $B_p$ is an unramified quadratic extension of $\Q_p$. Such an extension of $\Q_p$ is unique, up to an isomorphism. In other words, the field extension $B_p/\Q_p$ does not depend on a choice of $d$ up to an isomorphism.

In order to construct $K$, choose a prime $\ell \ne p$ such that  $\ell  \bmod p$ is {\sl not} a square in $\F_p$.
Let us consider the cubic  polynomial
$$f(x)=x(x^2- \ell) + \frac{p\ell}{(p\ell+1)^4} \in \Z_{(\ell)}[x] \bigcap  \Z_{(p)}[x] \subset \Q[x].$$
By Eisenstein's criterion over $\Z_{\ell}$, $f(x)$ is irreducible over $\Q_{\ell}$ and therefore over $\Q$. Now we may define the cubic field
$$K:=\Q[x]/f(x)\Q[x].$$
The reduction of $f(x)$ modulo $p$ coincides with the product
$x(x^2-\ell)$ of the linear polynomial $x$ and the irreducible quadratic polynomial $x^2-\ell$. By Hensel's lemma,  $f(x)$ splits over $\Q_p$ into a product
$$f(x)=f_1(x) f_2(x)$$
of a linear polynomial polynomial $f_1(x)$ and a quadratic polynomial $f_2(x)$; in addition, $f_2(x)$ splits into a product of two linear factors over an unramified quadratic extension of $\Q_p$. It follows that
$$K_p:=K \otimes_{\Q}\Q_p=\Q_p[x]/f_1(x)\Q_p[x] \oplus \Q_p[x]/f_2(x)\Q_p[x]=\Q_p\oplus \Q_p[x]/f_2(x)\Q_p[x]$$
where $\Q_p[x]/f_2(x)\Q_p[x]$ is an unramified quadratic extension of $\Q_p$. This implies that
$\Q_p[x]/f_2(x)\Q_p[x] \cong B_p$ and therefore
$$K_p \cong \Q_p \oplus B_p.$$
It remains to check that $K$ is totally real, i.e., all complex roots of $f(x)$ are real. Let $w$ be a complex root of $f(x)$. Then $f(w)=0$, i.e.,
$$w(w-\sqrt{\ell}) (w+\sqrt{\ell})=- \frac{p\ell}{(p\ell+1)^4}$$
and therefore
$$\min\{\mid w\mid, \mid(w-\sqrt{\ell}\mid, \mid(w+\sqrt{\ell}\mid\}<\frac{1}{p\ell+1}\le\frac{1}{3\cdot 2+1}=\frac{1}{7}.$$
This implies that $w$ lies in one of three circles of radius $1/7$ with real  centers $0, \sqrt{\ell}$ and $-\sqrt{\ell}$ respectively.
Since  the distance between any two centers is  $\ge \sqrt{\ell}>2/7$, these circles  do not meet each other. In particular, every root $w$ lies exactly in one of these circles.  On the other hand, let $w_1,w_2, w_3$ be the set of all roots of $f(x)$. If $a$ is any of the centers $0,\sqrt{\ell}, -\sqrt{\ell}$ then
$a(a^2-\ell)=0$ and therefore
$$\frac{p\ell}{(p\ell+1)^4}=f(a)=\prod_{i=1}^3 (a-w_i).$$
This implies that
$$\min\{ \mid a-w_i\mid, 1 \le i \le 3\}<\frac{1}{p\ell+1}\le \frac{1}{7}.$$
It follows that each of these circles contains, at least, one root of $f(x)$. We conclude that each circle contains exactly one root of $f(x)$.
Since each of these circles is stable under the complex conjugation, all the roots   of $f(x)$ must be real.  This proves that $K$ is totally real and we are done.

\section{Abelian fourfolds}
\label{dim4}

The following observation was inspired by results of Rutger Noot \cite[Prop. 4.1 on p. 165 and p. 168]{NootCrelle} about the reduction type of abelian varieties of Mumford's type \cite[Sect. 4]{MumfordMA}.

\begin{thm}
Let $X$ be an  abelian fourfold over $k$. Suppose that $k$ is sufficiently large with respect to $X$, $\rk(X)=3$ and $X$ enjoys  one of the following two properties.

\begin{itemize}
\item
$X$ is absolutely simple.
\item
$X$ is isogenous over $k$ to a product $X^{(3)}\times X^{(1)}$ of an (absolutely) simple abelian threefold $X^{(3)}$ and an ordinary elliptic curve $X^{(1)}$. \end{itemize}

Then at least one of the following two conditions holds.

\begin{itemize}
\item[(i)]
 there exist an imaginary quadratic field $B$ and an embedding $B \hookrightarrow \End^0(X)$ that sends $1$ to $1$.
\item[(ii)]
$X$ is not simple and $X^{(3)}$ is an almost ordinary abelian threefold that is not neat and therefore satisfies the conditions of Theorem \ref{main}. In particular, $\End^0(X^{(3)})$ contains an imaginary quadratic subfield.
\end{itemize}
\end{thm}

\begin{proof} We have $$\rk(X)=\dim(X)-1.$$
If $X$ is simple then it follows from Theorem 3.6 of \cite{LenstraZarhin} that
the condition (i) holds.

Now we may assume that $X=X^{(3)}\times X^{(1)}$.  Recall that $\End^0(X^{(1)})$ is an imaginary quadratic field and $\rk(X^{(1)})=1$.
 We have
$$\rk(X^{(3)}) \le \rk(X)=3 \le \rk(X^{(3)})+\rk(X^{(1)})=\rk(X^{(3)})+1.$$
This implies that $\rk(X^{(3)})=2$ or $3$.
If $\rk(X^{(3)})=3$ then all the roots of $\P_{X^{(3)}}(t)$ are simple and therefore $\End^0(X^{(3)})$ is a field (recall that $X^{(3)}$ is simple). It follows from Corollary \ref{essen211} (applied to $X=X^{(3)}$ and $Y=X^{(1)}$) that there is a field embedding $\End^0(X^{(1)})\hookrightarrow \End^0(X^{(3)})$ and one may take as $B$ the  field $\End^0(X^{(1)})$, which implies that the condition (i) holds.
If $\rk(X)=2$ and $\P_{X^{(3)}}(t)$ has no multiple roots (i.e., is irreducible) then $X^{(3)}$ is {\sl not} neat. It follows from Theorem \ref{main} that the condition (ii) holds. The only remaining case is when $\P_{X^{(3)}}(t)$ has multiple roots, i.e., $$\P_{X^{(3)}}(t)=\P_{X^{(3)},\min}(t)^d$$ where $d>1$ is an integer. Since among the roots of $\P_{X^{(3)},\min}$ there are no square roots of $q$, $\deg(\P_{X^{(3)},\min})$ is even. It follows that $d=3$ and $\P_{X^{(3)},\min}$ is a quadratic polynomial.  But then $R_X$ consists of two elements and $\Gamma^{\prime}(X^{(3)},k)$ is a cyclic group, i.e., its rank is $1$, which is not the case. The obtained contradiction ends the proof.
\end{proof}

\section{Corrigendum to \cite{ZarhinEssen}}
\label{corr}

\begin{itemize}
\item

Page 274, Remark 2.1 The displayed formula
 should read
$$\rk(\Gamma)\le\lfloor\deg(\mathcal{P}_{\min})/2\rfloor+1.$$
The formula on last line should read
$\lfloor\deg(\mathcal{P}_{\min})/2\rfloor+1$.

\item
Page 280, Theorem 2.12. The beginning of second sentence

{\sl The  equality
$$\rk(\Gamma(X\times Y))=\rk(X)+\rk(Y)-1$$ holds true if and only if
there exists an imaginary quadratic field $B$  enjoying the following properties:}

should read as follows.

{\sl If
$$\rk(X\times Y)=\rk(X)+\rk(Y)-1$$ then there exists an imaginary quadratic field $B$ that enjoys the following properties.}

\item
Page 281, Remark 3.1, last line. The formula should read
$$\rk(\Gamma)=\lfloor\deg(\mathcal{P}_{\min})/2\rfloor+1.$$

\item
Page 284,  line 8. $\alpha-1$ should read ${\alpha^{\prime}}^{-1}$.
\end{itemize}

\section{Corrigendum to \cite{ZarhinK3}}
\begin{itemize}
\item
Pages 267, 269 (and throughout the text),  $\angle$ and $\angle^{*}$ should read $L$ and $L^{*}$ respectively.
\item
Page 267, line -10: multiplicities should read multiplies.
\item
Page 271, Definition 3.4: ignore senseless {\bf tenibk}.
 \end{itemize}

\end{document}